\documentclass{amsart}
\usepackage{graphicx} 
\usepackage{xcolor}
\usepackage{enumitem}
\usepackage{tikz}
\usepackage{xfrac}
\usepackage{hyperref}
\usepackage{subcaption}
\usepackage{fancyheadings}
\usepackage{fancyhdr}
\usepackage{amsmath,amssymb,mathtools}
\usepackage[nobysame]{amsrefs}
\newtheorem{theorem}{Theorem}[section]

\theoremstyle{definition}
\newtheorem{definition}[theorem]{Definition}
\newtheorem{example}[theorem]{Example}

\theoremstyle{remark}
\newtheorem{remark}[theorem]{Remark}
\numberwithin{equation}{section}
\newcommand{\hl}{\frac12}
\newcommand{\mip}{\bar{y}}
\newcommand{\para}[1]{\paragraph{\em #1}}

\begin{document}

\title[B-stability of numerical integrators on Riemannian manifolds]{B-stability of numerical integrators on Riemannian manifolds}
\author[M. Arnold]{Martin Arnold }
\address{Institute of Mathematics, Martin Luther University Halle-Wittenberg, 06099 Halle (Saale), Germany}
\email{martin.arnold@mathematik.uni-halle.de}

\author[E. Celledoni]{Elena Celledoni}
\address{Department of Mathematical Sciences, Norwegian University of Science and Technology, 7034 Trondheim, Norway}
\email{elena.celledoni@ntnu.no}

\author[E. Çokaj]{Ergys Çokaj*}
\address{Department of Mathematical Sciences, Norwegian University of Science and Technology, 7034 Trondheim, Norway}
\email{ergys.cokaj@ntnu.no}
\thanks{* Corresponding author}
\thanks{This work was supported by the European Union’s Horizon 2020
research and innovation programme under the Marie Skłodowska-Curie
grant agreement No. 860124. This publication reflects only the author’s
view and the Research Executive Agency is not responsible for any use
that may be made of the information it contains.}

\author[B. Owren]{Brynjulf Owren}
\address{Department of Mathematical Sciences, Norwegian University of Science and Technology, 7034 Trondheim, Norway}
\email{brynjulf.owren@ntnu.no}

\author[D. Tumiotto]{Denise Tumiotto}
\address{Institute of Mathematics, Martin Luther University Halle-Wittenberg, 06099 Halle (Saale), Germany}
\email{denise.tumiotto@mathematik.uni-halle.de}


\begin{abstract}
We propose a generalization of nonlinear stability of numerical one-step integrators to Riemannian manifolds in the spirit of Butcher’s notion of B-stability. Taking inspiration from Simpson-Porco and Bullo, we introduce non-expansive systems on such manifolds and define B-stability of integrators. In this first exposition, we provide concrete results for a geodesic version of the Implicit Euler (GIE) scheme. We prove that the GIE method is B-stable on Riemannian manifolds with non-positive sectional curvature. We show through numerical examples that the GIE method is expansive when applied to a certain non-expansive vector field on the 2-sphere, and that the GIE method does not necessarily possess a unique solution for large enough step sizes. Finally, we derive a new improved global error estimate for general Lie group integrators.
\end{abstract}
\maketitle
\thispagestyle{fancy}
\lhead[]{This manuscript has been published on the Journal of Computational Dynamics\\
Vol. 11, No. 1, January 2024, pp. 92-107\\
\textcolor{blue}{\href{https://www.aimsciences.org//article/doi/10.3934/jcd.2024002}{doi:10.3934/jcd.2024002}}}
\rhead[]{ }

\section{Introduction}
Stability is a fundamental property of numerical methods for stiff nonlinear ordinary differential equations. 
It is important for controlling the growth of error in the numerical approximation and is used in combination with local error estimates to obtain bounds for the global error. Stability bounds can also in some situations be used to ensure the existence and uniqueness of a solution to the algebraic equations arising from implicit integrators.
In the literature, one can find a large variety of stability definitions for numerical integrators with various different aims. Some of them apply to linear test equations, others are of a more general nature and apply to nonlinear problems with certain prescribed properties. Most of the stability definitions found in the literature are developed for problems modeled on linear spaces.
In particular, there is a well-established non-linear stability theory, where an inner product norm is used to measure the distance between two solutions and the corresponding numerical approximations.
Pioneering contributions to this theory were made by Dahlquist and Butcher in the mid-1970s \cites{dahlquist1975,butcher1975}, in the wake of the legendary numerical analysis conference in Dundee, 1975. The notions of G-stability for multi-step methods \cite{dahlquist1975} and B-stability of Runge--Kutta methods \cite{butcher1975} were developed.
The overall idea of B-stability is that whenever the norm of the difference between two solutions of the ODE is monotonically non-increasing, the numerical method should exhibit a similar behavior, that is, the difference in norm between the two corresponding numerical solutions should not increase over a time step. Much is known about B-stable Runge--Kutta methods, and there is even an algebraic condition on the coefficients $(A,b)$ of a method that ensures its B-stability.
A key ingredient is the one-sided Lipschitz condition, also called a monotonicity condition, on the ODE vector field.
We refer the reader to the excellent monographs \cites{hairer96sod, dekker1984} for a detailed treatment of the various definitions of stability and B-stability in particular.

We remark that whether a particular ODE system is non-expansive depends on the choice of inner product norm, but the notion of a B-stable Runge--Kutta method does not, see \cite{hairer96sod}*{p. 182}.
In this paper, we shall be concerned with {\em unconditional stability}, meaning that step sizes $h\in(0,\infty)$ are 
allowed. This excludes all explicit integrators, and it makes it necessary to assume that both the flow of the ODE vector field and the numerical method map are well defined for all positive $t$. Dahlquist and Jeltsch \cite{dahlquist1979gdo} introduced 
{\em generalized disks of contractivity} in order to consider also the case in which limitations on the ODE vector field and the step size are imposed.

We shall here consider systems of ODEs whose solutions evolve on a smooth manifold. 
We are primarily interested in numerical integrators which are intrinsic, that are not developed for a particular choice of local coordinates, or based on a specific embedding of the manifold into an ambient space.
There are several such numerical methods available in the literature.

Crouch and Grossman \cite{crouch93nio} proposed to build integrators by composing flows of so-called frozen vector fields, and these methods were later extended to a more general format in \cite{celledoni03cfl} called {\em Commutator-free Lie group methods}.
Munthe--Kaas introduced numerical integrators for homogeneous spaces \cite{munthe-kaas99hor} by equipping the manifold with a left transitive Lie group action which was used together with the exponential map to transform the ODE vector field locally to a vector field on the underlying Lie algebra. Its flow is approximated by any classical Runge--Kutta method, and the result is mapped back to the manifold by composing the group action with the exponential map.

In computational mechanics there were early contributions to numerical integration on particular manifolds, such as the rotation group $SO(3)$ and the special Euclidean group $SE(3)$. A landmark paper in the design of conservative methods for Hamiltonian systems on Lie groups is the one by Lewis and Simo \cite{lewis94caf}.
For rod dynamics, an important paper was that of Simo and Vu-Quoc \cite{simo88otd} who developed a geometrically exact formulation for rods undergoing large motions, and for the time stepping they devised a version of the Newmark methods applicable to Lie groups. These methods can be generalized to the so-called $\alpha$-methods \cite{hilber1977ind} in a Lie group setting, see \cites{arnold07cot, arnold15eao}.
Parametrization of the manifold in question, such as the rotation group, plays a significant role in computational mechanics, for efficiency, accuracy, and storage requirements. When using (minimal) local coordinates for global simulation, one inevitably runs into problems with singularities, these issues have been studied and amended by several authors, e.g. \cites{terze16sft, holzinger21tio}.
Hamiltonian systems are often formulated on cotangent bundles, in which case symplectic integrators can be derived through the discretization of a variational problem, this approach is sometimes named discrete mechanics. The pioneering work by Marsden and West \cite{marsden01dma} developed this theory for Euclidean spaces, and it has later been generalized to Lie groups in a number of papers \cites{lee07lgv,bogfjellmo16hos,celledoni14ait,demoures15dvl, hall17lgs, hante2021, leitz18glg}.

Finally, on a Riemannian manifold, it is natural to base the numerical schemes primarily on the Riemannian exponential map.
Leimkuhler and Patrick \cite{leimkuhler96asi} derived a symplectic integrator for Riemannian manifolds, and in \cite{celledoni02aco} the authors suggest using Riemannian normal coordinates to define a retraction map.

For an in-depth account of Lie group methods, we refer to \cites{iserles2000,christiansen11tis, celledoni14ait, owren18lgi, celledoni21lgi} and references therein. 

In this paper we shall make the first attempt to generalize B-stability to Riemannian manifolds, replacing the inner product norm with the Riemannian distance function. 
We take inspiration from the work of Simpson-Porco and Bullo \cite{SIMPSONPORCO201474} who considered contraction properties of a continuous system.
In Section~\ref{sec:contr_sys} we define what we mean by a non-expansive system on a Riemannian manifold, and we 
state the definition of B-stability of a general numerical method in this setting.
Then, in Section~\ref{sec:num_int_on_manifolds} we first
present two examples of numerical methods: the geodesic versions of the implicit Euler method (GIE) and the implicit midpoint rule (GIMP).
Then we prove a B-stability result for the GIE method in the case that the manifold has non-positive sectional curvature.
We also provide numerical experiments for a particular vector field on the two-sphere, $S^2$, showing that neither the GIE nor the GIMP method is B-stable on this manifold which has positive sectional curvature. We briefly discuss also for this example a non-uniqueness issue with the GIE method which is different from what is known from the Euclidean setting.
Finally, in Section~\ref{sec:glerr} we present a bound for the global error of numerical methods, based on the monotonicity condition.

\section{Non-expansive systems}\label{sec:contr_sys}
We begin by briefly introducing some notation and terminology, mostly adhering to the monograph by Lee \cite{lee18itr}.
A Riemannian manifold is a pair $(M, g)$, where $M$ is a smooth manifold and $g$ is a smoothly varying inner product defined on each tangent space $T_pM,\;p\in M$. We will use interchangeably the notations $g(\cdot,\cdot)$ and $\langle\cdot,\cdot\rangle$.
Associated to $(M,g)$ is the Levi-Civita connection, the unique affine connection $\nabla$, which for any three vector fields $X, Y, Z$ on $M$ satisfies
$X \langle Y,Z\rangle=\langle \nabla_XY,Z\rangle+ \langle Y,\nabla_X Z\rangle$ and $[X,Y]=\nabla_XY-\nabla_YX$. 
The connection also defines the covariant derivative of vector fields along curves, we use the notation $D_tV(t)$ to denote the covariant derivative of $V(t)$ along $\gamma(t)$, see \cite{lee18itr}*{Theorem 4.21}.
A curve $\gamma: [a,b]\rightarrow M$ is geodesic if it satisfies the equation $D_t \dot{\gamma}(t)=0$ along $\gamma(t)$. 
A geodesic that connects two points $p$ and $q$ is called a geodesic segment.
If this second order differential equation, together with initial data $\gamma(0)=p\in M,\;\dot{\gamma}(0)=v_p\in T_pM$ yields a solution $\gamma(t)$, $t\in[0,t^*]$, 
thus $\exp_p: T_pM\rightarrow M$.
A similar notation is used for the $t$-flow, $\exp(tX)$, of a vector field $X$ on $M$, it is the diffeomorphism on $M$, $p\mapsto y(t)$
where $\dot{y}=X|_y,\ y(0)=p$, and its domain of definition may be $t$-dependent.
A numerical method on M is a map $\phi_{t,X}: M\rightarrow M$ that approximates the flow map $\exp(tX)$.
A set $\mathcal{U}\subseteq M$ is geodesically convex if, for each $p, q \in \mathcal{U}$, there is a unique minimizing geodesic segment from $p$ to $q$ contained entirely in $\mathcal{U}$. A vector field $X$ is forward complete on $\mathcal{U}$ if for every $p\in \mathcal{U}$, $\exp(tX)p$ is defined for all $t\geq 0$.
If for every $(t, p)\in [0, \infty) \times \mathcal{U}$ it holds that $\exp(tX)p\in \mathcal{U}$, we say that $\mathcal{U}$ is forward $X$-invariant. Similarly, for a mapping $\rho$ the set $\mathcal{U}$ is $\rho$-invariant if $\rho(y)\in\mathcal{U}$ for any $y\in \mathcal{U}$. We denote the length of a curve $\gamma:[a, b] \rightarrow M$ as $\ell(\gamma)=\int_a^b\|\dot{\gamma}(t)\| d t$, where $\|v\|:=\langle v, v\rangle^{\frac{1}{2}}$ is the $g$-norm.
The metric induces a distance function between pairs of points $p, q \in M$, $d(p, q)=\inf _{\gamma_{p \rightarrow q}} \ell \left(\gamma_{p \rightarrow q}\right)$, where $\gamma_{p\rightarrow q}$ is any continuous curve connecting $p$ and $q$. The following definition replaces the one-sided Lipschitz condition on a Riemannian manifold.
\begin{definition}\label{def:growth_bound}
Let $(M,\langle\cdot,\cdot\rangle)$ be a Riemannian manifold and let $\mathcal{U}\subset M$.   
We say that the vector field $X$ satisfies a monotonicity condition on the set $\mathcal{U}$ with constant $\nu\in\mathbb{R}$ if for every $x\in \mathcal{U}$ and $v_x \in T_x M$, it holds that 
\begin{equation}\label{eq:X_contractive_in_U}
\left\langle\nabla_{v_x} X, v_x\right\rangle \leq \nu\left\|v_x\right\|^2.
\end{equation} 
\end{definition}
Consider for every $x\in \mathcal{U}$, the linear operator $\left.\nabla X\right|_x:v_x \mapsto \nabla_{v_x}X$ on $T_xM$. 
The constant $\nu$ can be chosen as 
\begin{equation}\label{eq:nu_sup_mu_g}
    \nu = \sup_{x \in \mathcal{U}} \mu_g(\left.\nabla X\right|_x),
\end{equation}   
where $\mu_g$ is the logarithmic $g$-norm of $\left.\nabla X\right|_x$. 
For a linear operator $A: T_xM\rightarrow T_xM$, its logarithmic $g$-norm is defined as \cite{dekker1984}
$$
\mu_g(A) = \sup_{0\neq v\in T_xM}\frac{g(Av,v)}{g(v,v)}.
$$
In local coordinates $\mathbf{x}=(x_1,\ldots,x_m)$ on $M$, we write the vector field as $X=X^i(\mathbf{x})\partial_i$ and the metric tensor $g$ is represented by the matrix $\mathbf{g}(\mathbf{x})$ with elements $\mathbf{g}_{ij}=g(\partial_i,\partial_j)$. The operator $\nabla X$ has the matrix representation $\mathcal{A}(X)$
where $\mathcal{A}(X)^k_i = \partial_iX^k+\Gamma^k_{ij}X^j$ and where $\Gamma^k_{ij}$ are the Christoffel symbols of the connection.
We can now formulate the logarithmic $g$-norm of $\nabla X$ pointwise as 
$$
\mu_g(\nabla X) = \max \lambda\left[\mathbf{g}^{\sfrac{1}{2}}(\mathbf{x}) \mathcal{A}(X) \mathbf{g}^{-\sfrac{1}{2}}(\mathbf{x}) + \mathbf{g}^{-\sfrac{1}{2}}(\mathbf{x}) \mathcal{A}(X)^\top \mathbf{g}^{\sfrac{1}{2}}(\mathbf{x})\right],
$$
i.e., the largest eigenvalue of the matrix in square brackets, see also \cite{davydov22nec}.

\begin{theorem}\label{thm:bound_exact_flow}
Let $(M, g)$ be a Riemannian manifold,  $\mathcal{U}\subset M$  a geodesically convex set, and let $X$ be a vector field on $M$ satisfying the monotonicity condition \eqref{eq:X_contractive_in_U} on $\mathcal{U}$ with a constant  $\nu \in \mathbb{R}$. Suppose that for any $x_0, y_0 \in \mathcal{U}$, there is a $t^*>0$ such that $\exp (t X) x_0$ and $\exp (t X) y_0$ exist and are contained in $\mathcal{U}$ for every $t\in[0, t^*]$. 
Then,  it holds that
\begin{equation}\label{eq:distance_statement}
d\left(\exp (t X) x_0, \exp (t X) y_0\right) \leq  d\left(x_0, y_0\right)\mathrm{e}^{\nu t} \quad \text{for every }t\in[0, t^*].
\end{equation} 
\end{theorem}
\begin{remark}
The condition that the set $\mathcal{U}$ is geodesically convex can be weakened by introducing the notion of a $K$-reachable set as in \cite{SIMPSONPORCO201474}.
\end{remark}
\begin{proof}
The construction for the proof is illustrated in Figure \ref{fig:contr_image}. 
Since $\mathcal{U}$ is geodesically convex, there is a unique minimizing geodesic $\gamma(s)\in \mathcal{U}$
connecting $x_0, y_0 \in \mathcal{U}$, with $\gamma(0) = x_0$ and $\gamma(1) = y_0$. We will be using the notation 
$\Gamma(s, t) := \exp(tX)\gamma(s)$, as in 
\cite{lee18itr}*{Chapter 6}, and $\Gamma(s, t)$ is contained in $\mathcal{U}$.  For a fixed $t\in[0, t^*]$, consider the length $\ell(t)$ of the curve  $s \mapsto \Gamma(s, t)$, $s \in [0,1]$, that is 
\begin{equation}\label{eq:l_of_t_def}
    \ell(t) = \int_0^1 \langle \partial_s \Gamma(s,t), \partial_s \Gamma(s,t)\rangle^{\frac{1}{2}} ds,
\end{equation} 
and we have $d(x_0, y_0) = \ell(0)$. Let 
\begin{equation}\label{eq:S_T_short_notation}
    S(s,t):= \partial_s\Gamma(s,t), \quad T(s,t):= \partial_t\Gamma(s,t).
\end{equation} 
We will use that
\begin{equation*}   
D_t S(s, t)=D_s T(s, t)=D_sX|_{\Gamma(s,t)},
\end{equation*}   
following from the symmetry lemma \cite{lee18itr}*{Lemma 6.2}.
Differentiating with respect to $t$, using the chain rule and the properties of the Levi-Civita connection, we have
\begin{align*}
    \frac{\mathrm{d}\ell(t)}{\mathrm{d}t}  & =\int_0^1 \frac{\partial_t \langle S(s,t), S(s,t)\rangle}{2\|S(s,t)\|} ds = \int_0^1 \frac{\langle D_tS(s,t), S(s,t)\rangle}{\|S(s,t)\|} ds \\
    & = \int_0^1 \frac{\langle D_sT(s,t), S(s,t)\rangle}{\|S(s,t)\|} ds  = \int_0^1 \frac{\langle D_sX(\Gamma(s,t)), S(s,t)\rangle}{\|S(s,t)\|} ds  \\
    & \leq \int_0^1 \frac{\nu\langle S(s,t), S(s,t) \rangle}{\|S(s,t)\|} ds  = \nu \,\ell(t),
\end{align*}
where the last inequality follows from the assumption that  
$X$ satisfies the monotonicity condition \eqref{eq:X_contractive_in_U}.
By Gronwall's lemma, we obtain the inequality
\begin{equation*}   \label{eq:gronwall_lemma}
    \ell(t)\leq \ell(0)\mathrm{e}^{\nu t}, \quad \textrm{ for each } t\in[0, t^*],
\end{equation*} 
and conclude that
\begin{equation*}   \label{eq:dist_length}
    d(\exp(tX)x_0, \exp(tX)y_0)\leq \ell(t) \leq \ell(0) \mathrm{e}^{\nu t} = d(x_0, y_0) \mathrm{e}^{\nu t}.
\end{equation*} 
\end{proof}

\begin{remark}
Choosing $\nu = \sup\{\|\nabla X|_{p}\|:p \in \Gamma([0, t^*]\times[0,1])\}$ leads to a bound similar to the one in Theorem 1.2 by Kunzinger et al. in \cite{kunzinger06gge}.
\end{remark}

\begin{figure}
    \includegraphics[trim = 0 30 0 0 ]{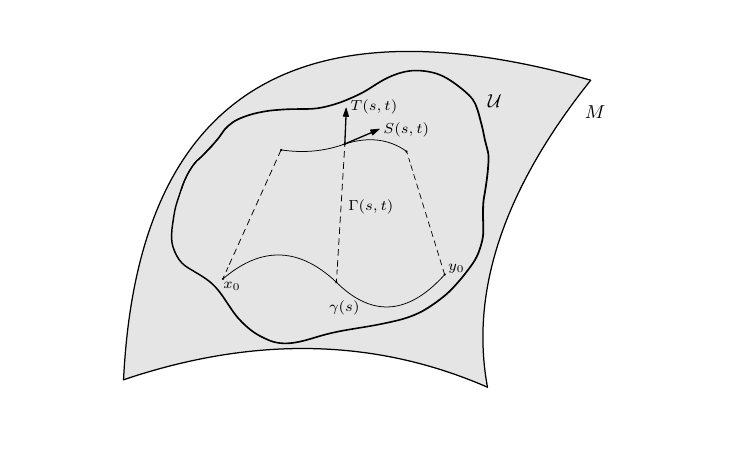}
    \caption{Construction for the proofs of Theorems \ref{thm:bound_exact_flow} and \ref{thm:b_stab_gie}.}
    \label{fig:contr_image}
\end{figure}

The next definition is inspired by the definition of contracting systems by Simpson-Porco and Bullo in \cite{SIMPSONPORCO201474}.
\begin{definition}[Non-expansive system]
    Let $(M,g)$ be a Riemannian manifold. Let $\mathcal{U} \subseteq M$ be an open, geodesically convex set and $X \in \mathfrak{X}(M)$. If
    \begin{enumerate}[label=(\roman*)]
        \item $X$ is forward complete on $\mathcal{U}$,
        \item $\mathcal{U}$ is forward $X$-invariant,
        \item $X$ satisfies the monotonicity condition \eqref{eq:X_contractive_in_U} on $\mathcal{U}$ with $\nu \leq 0$,
    \end{enumerate}
the quadruple $(\mathcal{U}, X, g, \nu)$ is called a non-expansive system.
\end{definition}

We are now ready to give the definition of a B-stable numerical method on Riemannian manifolds.

\begin{definition}[B-stability]\label{def:b_stability}

Let $(M,g)$ be a Riemannian manifold and let $\phi_{h, X}$ be a numerical method on $M$. Suppose that for any non-expansive
system $(\mathcal{U}, X, g, \nu)$ on $M$, it holds that
\begin{enumerate}[label=(\roman*)]
\item   $\phi_{h, X}$ is forward complete on $\mathcal{U}$, i.e., $\phi_{h,X}$ is well defined for all $h>0$, and\label{item:method_forw_complete_on_u}
    \item $\mathcal{U}$ is forward  $\phi_{h, X}$-invariant for all $h>0$.
 \label{item:u_forw_method_invariant}
\end{enumerate}
If 
\begin{equation*}   \label{eq:b_stab}
    d(\phi_{h, X}(x_0), \phi_{h, X}(y_0)) \leq d(x_0, y_0), \quad x_0, y_0 \in \mathcal{U}, h>0,
\end{equation*} 
then $\phi_{h,X}$ is called B-stable.
\end{definition}

\section{Numerical integrators on manifolds and B-stability}\label{sec:num_int_on_manifolds} 
\para{Geodesic Explicit Euler (GEE) method}
The simplest numerical method defined on a Riemannian manifold is the Geodesic Explicit Euler method
\begin{equation}\label{eq:gee_method}
    y_{n+1} = \exp_{y_{n}}\left(hX|_{y_{n}}\right),
\end{equation} 
that can not be unconditionally stable, but will be used for comparison in the numerical experiments in Example 3.2.
\para{Geodesic Implicit Euler (GIE) method}
We consider the following definition of the Implicit Euler method in a Riemannian manifold 
\begin{equation}\label{eq:gie_method}
    y_n = \exp_{y_{n+1}}\left(-hX|_{y_{n+1}}\right).
\end{equation} 
This reduces to the classical implicit Euler method
when the manifold is the Euclidean space.
\para{Geodesic Implicit Midpoint (GIMP) method}
Similarly, we consider the implicit midpoint rule on a Riemannian manifold:
\begin{equation}\label{eq:gimp}
\begin{aligned}
    y_n &= \exp_{\mip}\left(-\hl h X|_{\mip}\right), \\
    y_{n+1} &= \exp_{\mip}\left(\hl h X|_{\mip}\right).
\end{aligned}
\end{equation} 
This method can be found in Zanna et al. \cite{zanna01aas} for the case of Lie group integrators.
It is a symmetric method, but it is not generally symplectic. In \cite{mclachlan17amv} a symplectic method was found for products of 2-spheres, $(S^2)^d$, that happens to be a time reparametrization of \eqref{eq:gimp}. It is called the {\em spherical midpoint method} (SPHMP). Applied to a single copy of $S^2$ it reads in Cartesian coordinates
\begin{equation} \label{eq:sphmp}
y_{n+1} = y_n + h X|_{\mip},\quad \mip=\frac{y_n+y_{n+1}}{\|y_n+y_{n+1}\|}.
\end{equation} 

\subsection{The case with non-positive sectional curvature}\label{sec:b_stab_gie}

In the next theorem, we prove the B-stability of the GIE method on Hadamard manifolds, i.e., manifolds with non-positive sectional curvature.
\begin{theorem}[B-stability of the GIE method]\label{thm:b_stab_gie}
    Let $(M, g)$ be a Riemannian manifold with non-positive sectional curvature. Then, the GIE method \eqref{eq:gie_method} is B-stable.
\end{theorem}
\begin{proof}
Let $(\mathcal{U}, X, g, 0)$ be a non-expansive system of ODEs, and consider $\phi_{h,X}$ with step size $h>0$.  Let $\gamma_0(s), s\in [0,1]$ be a curve in $\mathcal{U}$ such that $\gamma_0(0)=x_0\in \mathcal{U}$ and $\gamma_0(1)=y_0\in \mathcal{U}$, and set $\gamma_1(s)=\phi_{h, X}\left(\gamma_0(s)\right)$. By assumption $\gamma_1(s)$ is well defined and contained in $\mathcal{U}$. Consider the one-parameter family of curves  
\begin{equation*}   \label{eq:gamma_exp_y1}
\Gamma(s, t):=\exp _{\gamma_1(s)}(-t h X|_{\gamma_1(s)}).
\end{equation*} 
We have $\gamma_1(s)=\Gamma(s, 0)$ and $\gamma_0(s)=\Gamma(s, 1)$. 
Now, using as earlier the notation $S(s,t):= \partial_s\Gamma(s,t), \quad T(s,t):= \partial_t\Gamma(s,t)$, we have
\begin{equation*}   
\ell(t)=\int_0^1 \langle S(s,t), S(s,t)\rangle^{\frac{1}{2}} ds \quad \text { and } \quad
\frac{d\ell}{dt}(t) =\int_0^1 \frac{\partial_t \langle S(s,t), S(s,t)\rangle}{2\|S(s,t)\|} ds.
\end{equation*} 
Let $f(t)=\frac{1}{2} \partial_t\langle S(s, t), S(s, t)\rangle=\left\langle D_t S(s, t), S(s, t)\right\rangle$. We differentiate with respect to $t$ and apply the Jacobi equation together with the definition of sectional curvature and obtain
\begin{equation}
\begin{aligned}
\frac{df}{dt}(t) & =\left\langle D_t^2 S(s, t), S(s, t)\right\rangle+\left\|D_t S(s, t)\right\|^2 \label{eq:fprime} \\
& = -\left\langle R\big(S(s, t), T(s, t)\big) T(s, t), S(s, t)\right\rangle+\left\|D_t S(s, t)\right\|^2  \\
& = - K(S, T)\left(\left\|S\right\|^2\left\| T\right\|^2-\langle S, T\rangle^2\right) + \left\|D_t S(s, t)\right\|^2. 
\end{aligned}
\end{equation}  
Here $R$ is the Riemannian curvature tensor and $K$ is the sectional curvature. Since by assumption $K(S, T) \leq 0$ it follows that $\frac{df}{dt}(t) \geq 0$ for $t \in [0,1]$. By the symmetry lemma \cite{lee18itr}*{Lemma 6.2}, we get
\begin{equation*}   \label{eq:use_sym_lemma}
D_t S(s, t)|_{t=0}=D_s T(s, 0)=-h D_s X|_{\gamma_1(s)}.
\end{equation*} 
Then \begin{equation*}   \label{eq:fof0}
f(0)=-h\left\langle D_s X|_{\gamma_1(s)}, S(s, 0)\right\rangle \geq 0,
\end{equation*} 
since $X$ satisfies the monotonicity condition with $\nu=0$. So, we have 
\begin{equation*}   \label{eq:foftgeq0}
f(t)=f(0)+\int_0^t \frac{df}{d\tau}(\tau)\, d \tau \geq 0,
\end{equation*} 
which allows us to conclude that $\frac{d\ell}{dt}(t) \geq 0$.
Thus
\begin{equation}\label{eq:lof1_greater_lof0}
\operatorname{length}\left(\gamma_1(s)\right) = \ell(0)\leq\ell(1)=\operatorname{length}\left(\gamma_0(s)\right).
\end{equation} 
For any given $\varepsilon > 0$, we have $\ell(1) < d(\gamma_0(0), \gamma_0(1)) + \varepsilon$. By \eqref{eq:lof1_greater_lof0} and the definition of distance we obtain
\begin{equation*}   \label{eq:dist_length_num_flow}
    d(\gamma_1(0), \gamma_1(1))\leq \ell(0) \leq \ell(1)\leq d(\gamma_0(0), \gamma_0(1)) + \varepsilon.
\end{equation*} 
Since $\varepsilon$ is arbitrary, the condition for B-stability is satisfied.
\end{proof}

\begin{example}
\begin{figure}[ht]
    \centering
    \includegraphics[width=\textwidth]{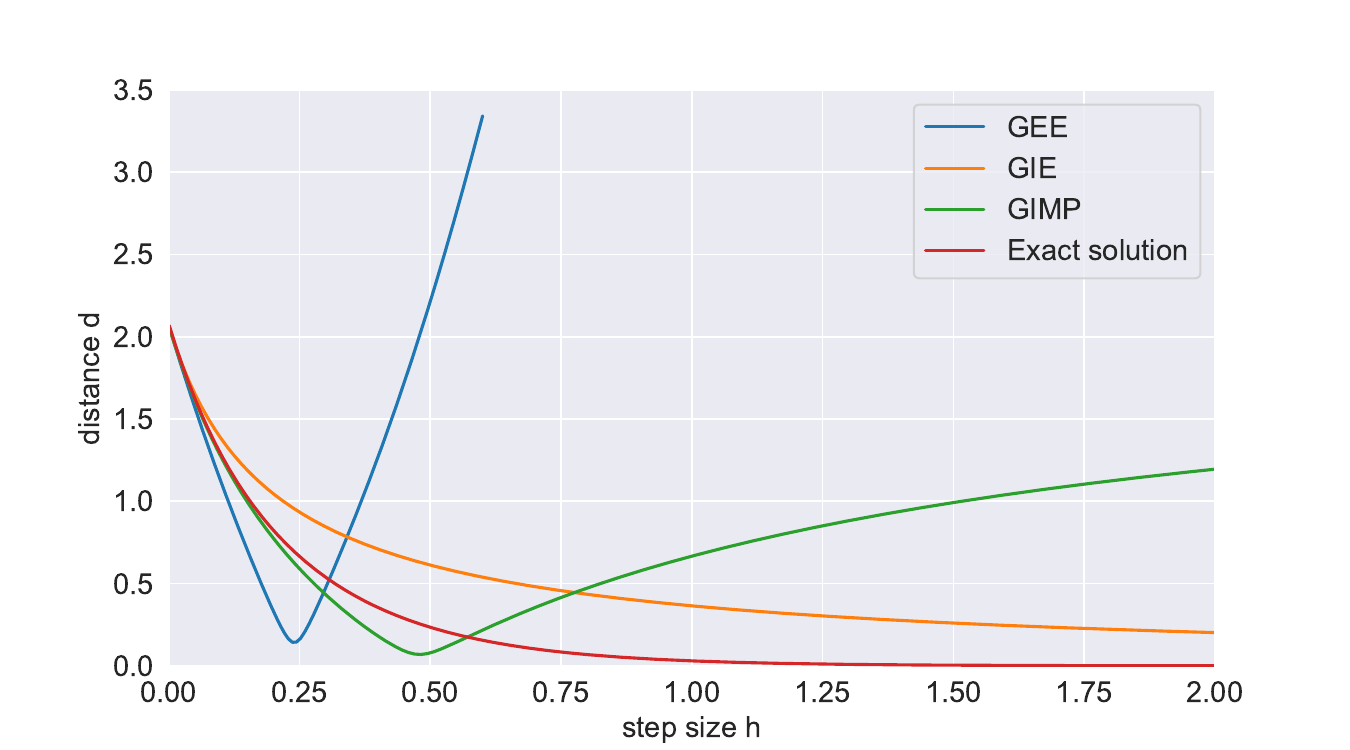}
    \caption{Riemannian distance of two solutions after one step plotted for increasing values of the step size $h$ with the same initial values. }
    \label{fig:SPD}
\end{figure}[$\mathbb{S}^n_{++}$]\label{example:dti}
The space $\mathbb{S}^n_{++}$ of symmetric positive definite matrices is a well-known example of a manifold with negative sectional curvature. Its tangent space at a point $A$, denoted by $T_A\mathbb{S}^n_{++}$, can be
identified as the set of $n \times n$ symmetric matrices. $\mathbb{S}^n_{++}$ is equipped pointwise with the metric 
\begin{equation}\label{eq:metric_sn++}
g_A(U, V) = \text{trace}\left(A^{-1}UA^{-1}V\right),
\end{equation} 
where $A\in \mathbb{S}^n_{++}$ and $U, V \in T_A\mathbb{S}^n_{++}$, \cites{sill1964symplectic, pennec2006riemannian}. 

The manifold $\mathbb{S}^n_{++}$ can be used as a model space for simple beam models, such as the Elastica \cite{erchuan2019riemannian}, or in diffusion tensor magnetic resonance imaging (DT-MRI) \cites{cheng2012efficient, rathi2007segmenting, celledoni18dissipative, fletcher2007riemannian}, via 3D tensors, i.e., $3\times 3$ SPD matrices. Another interesting application is the segmentation and recognition of images and videos represented by SPD matrices, \cites{arandjelovic2005face, huang2015face, tuzel2006region}.
Such applications usually involve averaging SPD matrices, for example, to collect noisy measurements of the object under consideration. In $\mathbb{S}^n_{++}$, a suitable mean was proposed by Karcher \cite{karcher1977riemannian}. Given $k$ matrices $Y_1, \ldots, Y_k \in \mathbb{S}^n_{++}$, we search for a matrix $X^*\in \mathbb{S}^n_{++}$, the Karcher mean,  such that
\begin{equation}\label{eq:karcher_mean}
    X^*=\underset{X \in \mathbb{S}^n_{++}}{\arg \min }\frac{1}{2} \sum_{j=1}^k d^2\left(X, Y_j\right),
\end{equation} 
i.e., $X^*$ is such that $\text{grad}\frac{1}{2} \sum_{j=1}^k d^2\left(X, Y_j\right) = 0$. Here, $d(X, Y)$ is the Riemannian distance between $X$ and $Y$ given as \cite{gregorio2013proximal}
\begin{equation}\label{eq:dist_sn++}
    d(X, Y) = \sqrt{\sum_{i = 1}^n \log^2\left(\lambda_i\left(X^{-\frac{1}{2}}YX^{-\frac{1}{2}}\right)\right)}, 
\end{equation} 
with $\lambda_i\left(X^{-\frac{1}{2}}YX^{-\frac{1}{2}}\right)$ being the $i^{\text{th}}$ eigenvalue of $X^{-\frac{1}{2}}YX^{-\frac{1}{2}}, i = 1, \ldots, n$, and $\text{grad}$ is the Riemannian gradient found e.g. in \cite{gregorio2013proximal}*{Lemma 2}  \begin{equation}\label{eq:grad_karcher_mean}
    \operatorname{grad}\frac{1}{2}d^2(X, Y)|_X=-X^{\frac{1}{2}} \log \left(X^{-\frac{1}{2}} Y X^{-\frac{1}{2}}\right) X^{\frac{1}{2}}.
\end{equation} 
For $\mathbb{S}_n^{++}$, the exponential map is explicitly known in terms of the matrix exponential and matrix square roots as
\begin{equation*}   
    \exp_A(tV) = A^{\frac{1}{2}}\textrm{e}^{t A^{-\frac{1}{2}}VA^{-\frac{1}{2}}}A^{\frac{1}{2}},
\end{equation*} 
for $A\in \mathbb{S}^n_{++}$ and $V \in T_A\mathbb{S}^n_{++}.$
The objective function $f(X) = \frac{1}{2}\sum_{j=1}^k d^2\left(X, Y_j\right)$ is defined as the geometric mean of symmetric positive definite matrices in \cite{moakher2005differential} and \cite{bhatia2006riemannian}, and is known to have a unique minimizer $X^*$ as in \eqref{eq:karcher_mean}, \cite{karcher1977riemannian}. 
There is no known closed-form solution for \eqref{eq:karcher_mean} and usually,  iterative methods are used to compute the Karcher mean.\\
In Figure~\ref{fig:SPD}, the Riemannian distance of two solutions after one step is plotted for increasing values of the step size $h$ with the same pair of initial values. One can observe the non-expansive behavior of the GIE and the GIMP method and the expansive behavior of the Geodesic Explicit Euler (GEE) method. The GEE solution is discontinued at $h=0.6$ for presentation purposes. The exact solution is calculated with strict tolerance by \texttt{odeint} of \texttt{scipy.integrate} in \texttt{Python}.
\end{example}

\subsection{The case with positive sectional curvature: The 2-sphere}\label{sec:2-sphere}

In this section, we consider systems on the 2-sphere $S^2$ with the standard metric. We show through an example that the GIE and GIMP methods fail to be B-stable. 

\subsubsection{Killing vector fields}\label{subsec:Killing}
A {\em Killing vector field} is a vector field $X$ such that the Lie derivative $\mathcal{L}_X g = 0$. This implies that
\begin{align*}
 0 &= (\mathcal{L}_Xg)(Y,Z) = X\langle Y,Z\rangle - \langle \mathcal{L}_XY,Z\rangle - \langle Y, \mathcal{L}_XZ\rangle   \\
 &= \langle \nabla_XY, Z\rangle + \langle Y,\nabla_XZ\rangle - \langle\nabla_XY-\nabla_Y X, Z\rangle
 - \langle Y, \nabla_XZ-\nabla_ZX\rangle\\
 &= \langle \nabla_Y X, Z\rangle +  \langle \nabla_Z X, Y\rangle,
\end{align*}
so that the monotonicity condition \eqref{eq:X_contractive_in_U} holds with $\nu=0$ for any such vector field.
In this sense one could say that the Killing vector fields represent a borderline case for non-expansive systems.
\subsubsection{A Killing vector field on $S^2$}\label{subsec:KillingS2}
Consider the vector field $X(y) = e_3 \times y$, which describes rotations on the 2-sphere around the $z$-axis. Using Cartesian coordinates, the GIE method \eqref{eq:gie_method} on the 2-sphere takes the form
\begin{equation}\label{eq:gie_sphere}
    y_0 = \exp_{y_1}(-hX|_{y_1}) = \cos \alpha \cdot y_1 - \frac{\sin \alpha}{\alpha} \cdot (hX|_{y_1}), \quad \alpha = \| -hX|_{y_1}\|.
\end{equation} 
We apply \eqref{eq:gie_sphere} to two initial points lying on the open northern hemisphere and measure the distance between the points for increasing values of the time step. The distance between two points $y_0, z_0\in S^2$  is calculated as
\begin{equation}\label{eq:dist_2_sphere}
    d(y_0, z_0) = \arccos{(y_0\cdot z_0)}.
\end{equation} 
\begin{figure}[ht]
\includegraphics[trim = 20 40 22 60, width=0.42\textwidth]{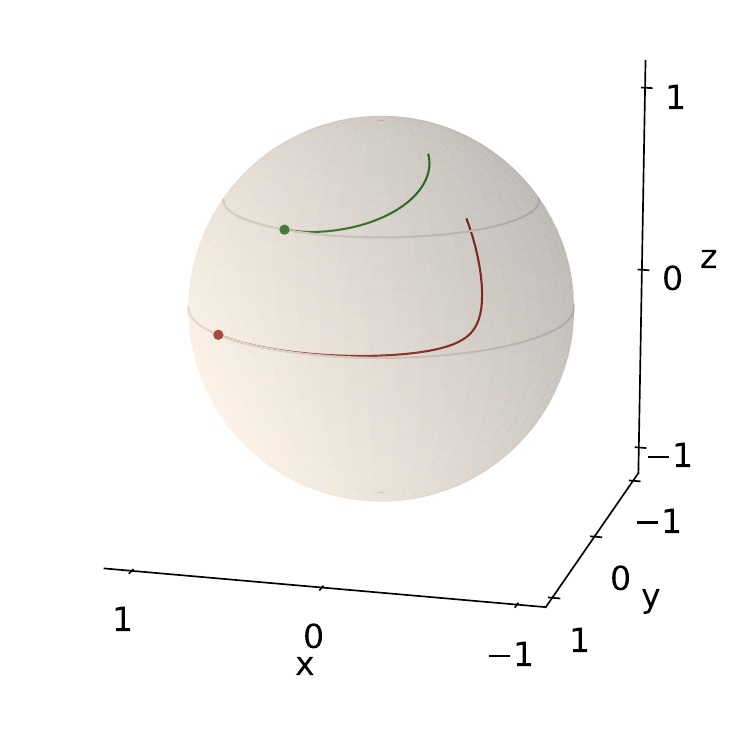}
\includegraphics[trim = 0 0 25 25, width = 0.57\textwidth]{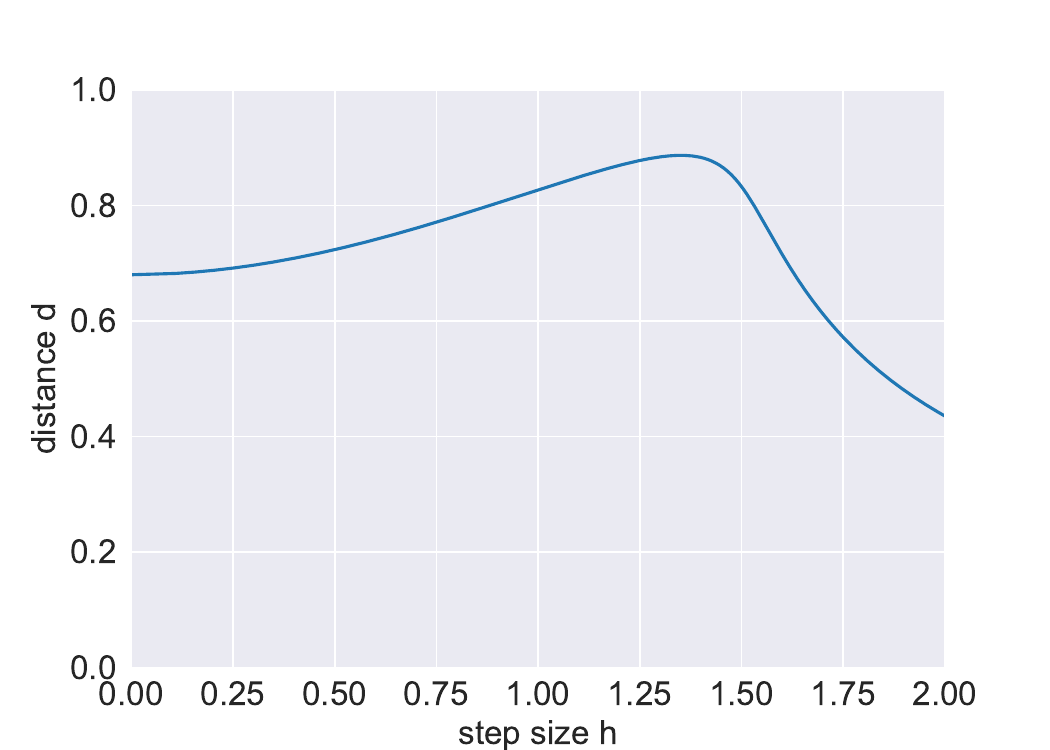}
    \caption{One step of GIE method for two initial points with increasing step size $h$ (left) and their Riemannian distance (right).}
\label{fig:2_trajectories}
\end{figure}
Figure \ref{fig:2_trajectories} (left) shows one step performed with the GIE method starting from two initial points with increasing step size $h$. In Figure \ref{fig:2_trajectories} (right), the distance between the trajectories is shown as a function of $h$. As can be seen from the distance curve, the GIE method shows an expansive behavior, and it is in fact small values of the step size that cause problems. In Figure \ref{fig:sphmp_gimp_spheres}, the SPHMP and the GIMP methods are tested on the same vector field. Both methods are a reparametrization of the exact solution for this problem.

\begin{figure}[ht]
    \centering
    \includegraphics[width=0.49\textwidth]{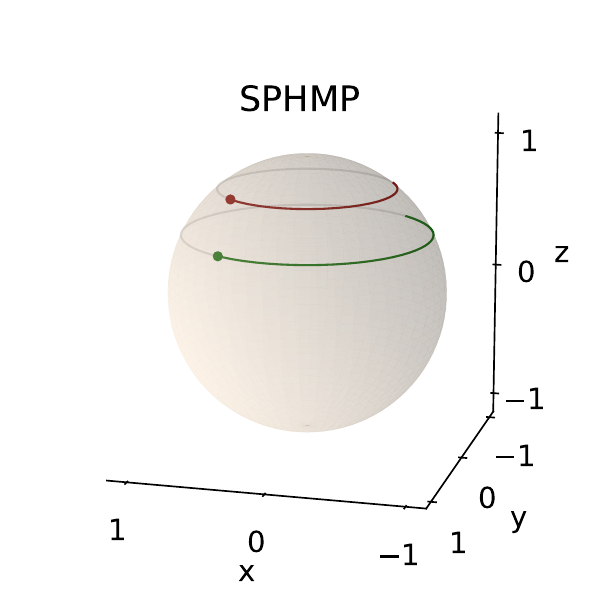}
    \includegraphics[width=0.49\textwidth]{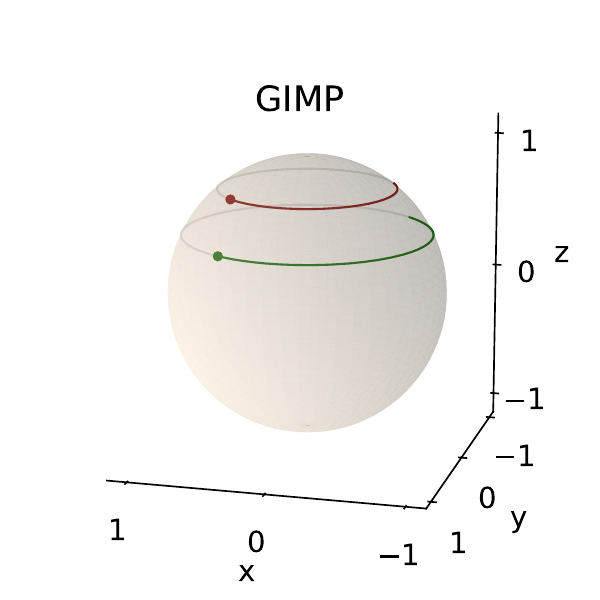}
    \includegraphics[trim = 20 0 40 0, width=\textwidth]{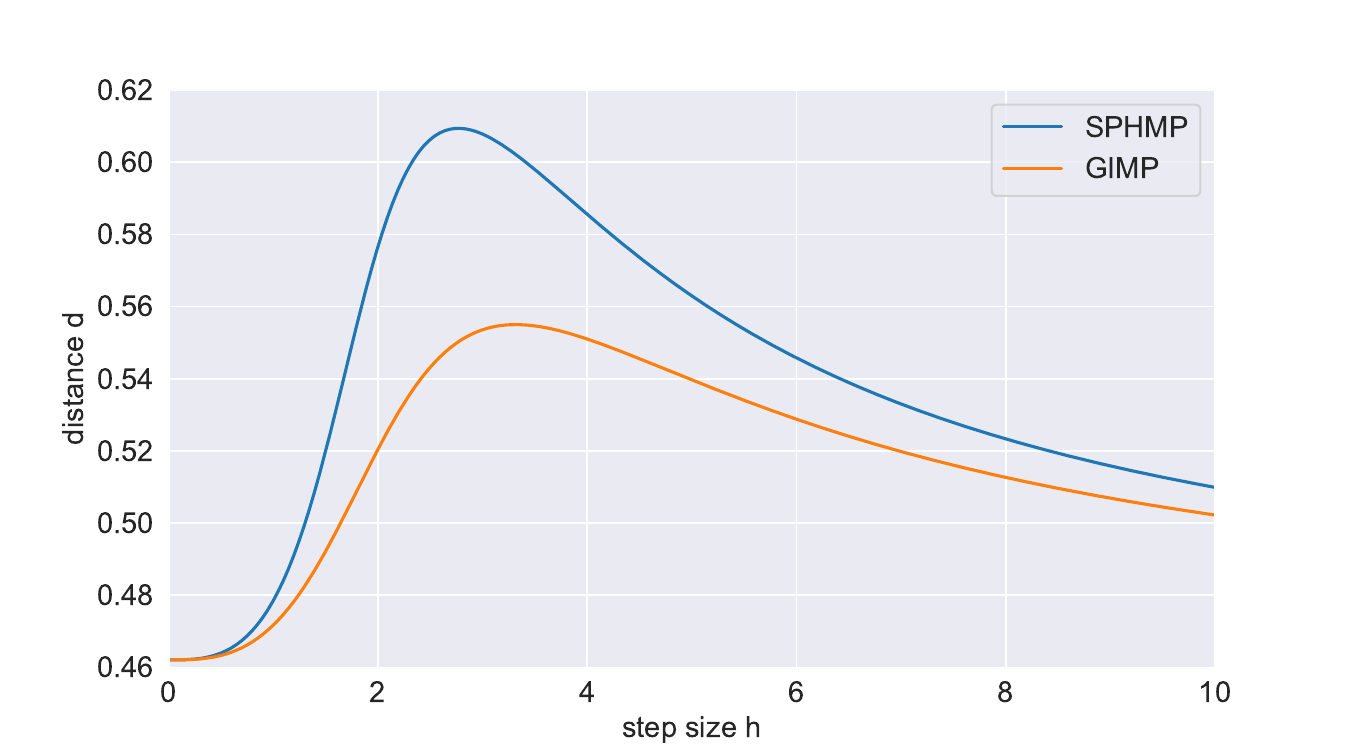}
    \caption{Top: One step of SPHMP \eqref{eq:sphmp} (left) and  GIMP \eqref{eq:gimp} (right) method for the same two initial points with increasing step size $h$. Bottom: Riemannian distance of two numerical solutions after one step plotted for increasing values of the step size $h$.}
    \label{fig:sphmp_gimp_spheres}
\end{figure}

\para{A non-uniqueness issue}
It is well-known from the theory of implicit Runge--Kutta methods that the conditions for the uniqueness of the solution to the implicit equations that must be solved in each time step involve the one-sided Lipschitz condition.
In the monograph by Hairer and Wanner \cite{hairer96sod} a precise result is given, and we include it here for completeness.
\begin{theorem}[Theorem 14.4 in \cite{hairer96sod}]\label{thm:hw_thm-14-4}
Consider a differential equation satisfying a one-sided Lipschitz condition with constant $\nu$. If the Runge--Kutta matrix $A$ is invertible and $h\nu < \alpha_0(A^{-1})$, then the system of equations to be solved in each time step possesses at most one solution.
\end{theorem}
We note that $\alpha_0$ is a function that depends only on the Runge--Kutta coefficients, and it is known that 
$\alpha_0(A^{-1})=1$ for the implicit Euler method. Thus, for $\nu\leq 0$ there is a unique solution for every $h>0$.
But the Killing vector field example on $S^2$ shows that this result is not generally true in Riemannian manifolds.
In fact, for this example, we see from \eqref{eq:gie_sphere} that the last component is decoupled from the other two. Writing for simplicity $y_0^3=: z_0$ and $y_1^3=:z$ we need to solve the scalar equation
\begin{equation} \label{eq:y3eq}
z_0 = \cos\left(h\sqrt{1-z^2}\right) z =: q\left(z,h\right)
\end{equation} 
with respect to $z$. One has $q\left(0,h\right)=0$ for all $h$, and $q\left(z_k,h\right)=0$ for $z_{k}=\pm \sqrt{1-\left(\frac{\pi}{h}\right)^2\left(\frac12+k\right)^2}$ for any $k\in\mathbb{N}$ such that $\left(\frac{\pi}{h}\right)^2\left(\frac12+k\right)^2\leq 1$.
In fact, for $h\in I_0=\left(0, \frac{\pi}{2}\right]$, $q\left(z,h\right)$ has precisely one zero, and for 
$h\in I_m =\left(\left(2m-1\right)\frac{\pi}{2}, \left(2m+1\right)\frac{\pi}{2}\right]$, $m\geq 1$, $q\left(z,h\right)$ has $2m+1$ zeros in $[-1,1]$.
All the zeros are simple and therefore there is a sign change in $q(z,h)$ at each of them. It follows that  $\exists\epsilon>0$ such that
if $h\in I_m$ and $|z_0|<\epsilon$, then  \eqref{eq:y3eq} has at least $2m+1$ solutions. One easily verifies that for each of these values of the last component, there is a unique solution for the first two components. We illustrate the structure of the solution in Figure~\ref{fig:bifurcation}. 

\begin{figure}[ht]
    \centering
    \includegraphics[width = \textwidth]{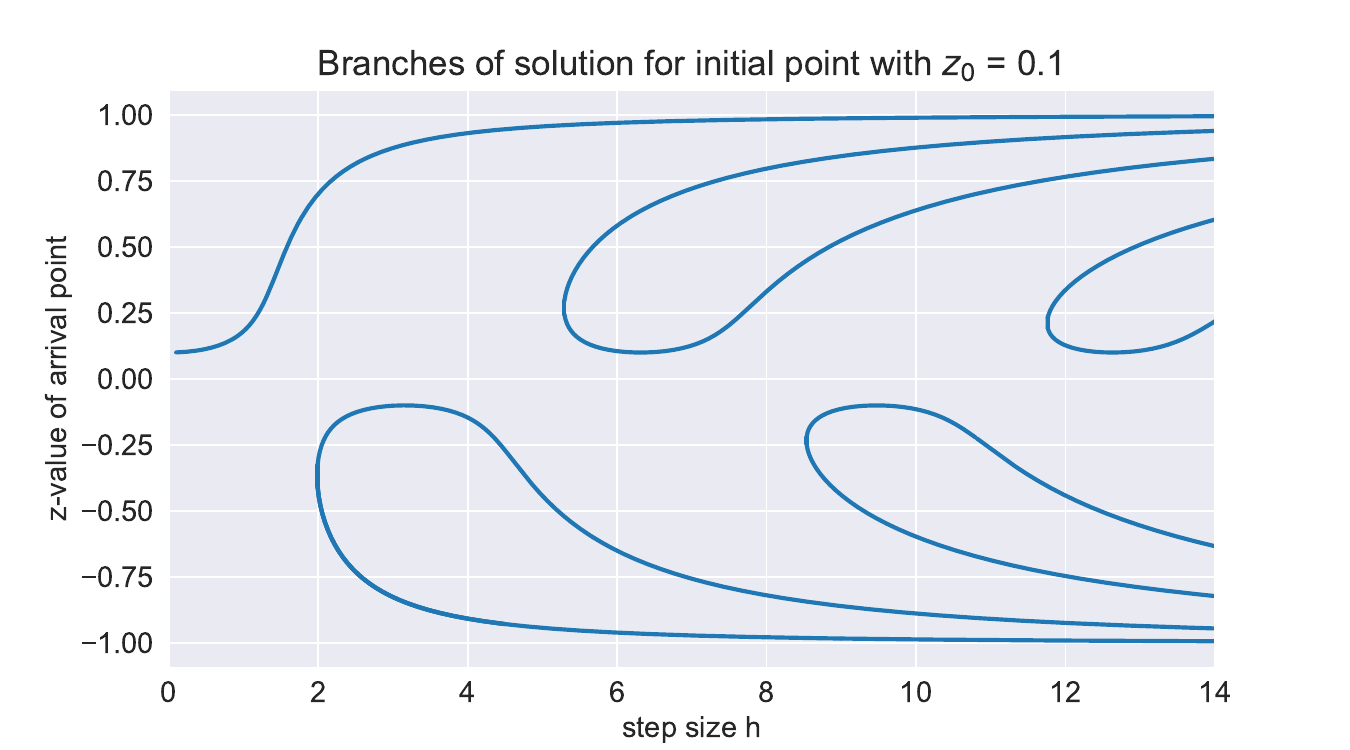}
    \caption{A bifurcation diagram for solutions to the equation \eqref{eq:y3eq}.}
    \label{fig:bifurcation}
\end{figure}

\subsubsection{Relation to other Lie group integrators}
For some homogeneous manifolds $M=G/H$, with $H$ a closed Lie subgroup of the Lie group $G$, the Geodesic Implicit Euler method \eqref{eq:gie_method} is equivalent to the implicit Lie-Euler method for a specific choice of isotropy, \cites{munthe-kaas99hor,celledoni14ait}, i.e., of the map $a: G/H\rightarrow \mathfrak{g}$ which is used to define the Lie group method:
\begin{equation}\label{eq:imp_lie_euler_method}
    y_{n+1} = \exp(ha(y_{n+1}))\,y_{n},
\end{equation} 
with $\mathfrak{g}$ the Lie algebra of $G$ and  $\exp:\mathfrak{g} \rightarrow G$ the Lie group exponential. See \cite{iserles2000} for an introduction to Lie group methods. The following example on $S^2$ illustrates the impact of the choice of isotropy on the approximation of the solution obtained via \eqref{eq:imp_lie_euler_method}.

\begin{example}
Consider a vector field on the 2-sphere $S^2=SO(3)/SO(2)$. 
In Cartesian coordinates, embedding $S^2$ in $\mathbb{R}^3$, the ODE  can be written as 
\begin{equation}\label{eq:gr_fl_sys}
    \dot{y} = a(y) \times y,
\end{equation} 
where $\times$ denotes the vector cross product.
By the identification of $(\mathbb{R}^3, \times)$ with the Lie algebra $\mathfrak{so}(3)$, we have that
$a: S^2 \rightarrow \mathbb{R}^3\simeq \mathfrak{so}(3)$. The action of the Lie group exponential $\exp: \mathbb{R}^3\simeq \mathfrak{so}(3)\rightarrow SO(3)$ on a vector $p\in\mathbb{R}^3$ takes the simple form:
$$\exp(a)p=p+\frac{\sin(\alpha)}{\alpha}\,a\times p-\frac{1-\cos(\alpha)}{\alpha^2}\,a\times (a \times p) , \qquad \alpha=\|a\|, \qquad a\in\mathbb{R}^3\simeq\mathfrak{so}(3).$$
We remark that for a given vector field $X(y)=a(y)\times y$, the choice of $a(y)$  is not unique. In fact, we can replace $a(y)$ with its projection orthogonal to $y$ without changing $X(y)$, and similarly
replacing $a(y)$ by $a(y) + c(y)\,y$, with 
$ c: S^2\rightarrow \mathbb{R}$, 
does not alter $X(y)$:
\[
\dot{y} = a(y) \times y = (a(y) + c(y)y) \times y, \qquad y^\top a(y)=0.
\]
On the other hand, the numerical approximation obtained by the method \eqref{eq:imp_lie_euler_method},
\[ 
y_{n+1} = \exp(h(a(y_{n+1}) + c(y_{n+1})\,y_{n+1}))y_{n},
\]
does depend on the choice of $c(y)$, see also Figure~\ref{fig:isotropy_plot_different_c_values}. Similarly, we cannot expect that in general different Lie group integrators have the same stability behavior when applied to the same vector field $X$.\\
In Figure~\ref{fig:isotropy_plot_different_c_values} we illustrate the isotropy issue by applying the Implicit Lie--Euler method
to the problem
\begin{equation}\label{eq:iso_eq}
    \dot{y} = e_3\times y = (e_3 + (c-1) y_3\,y )\times y
\end{equation} 
for $c\in[-2,2]$ and step size $h=2$. This means that $c=0$ corresponds to the GIE method, whereas for $c=1$ the exact solution is reproduced. We observe that the difference in solutions may expand, contract or stay constant, depending on the choice of isotropy parameter $c$.

\begin{figure}[ht]
    \includegraphics[width = 0.42\textwidth]{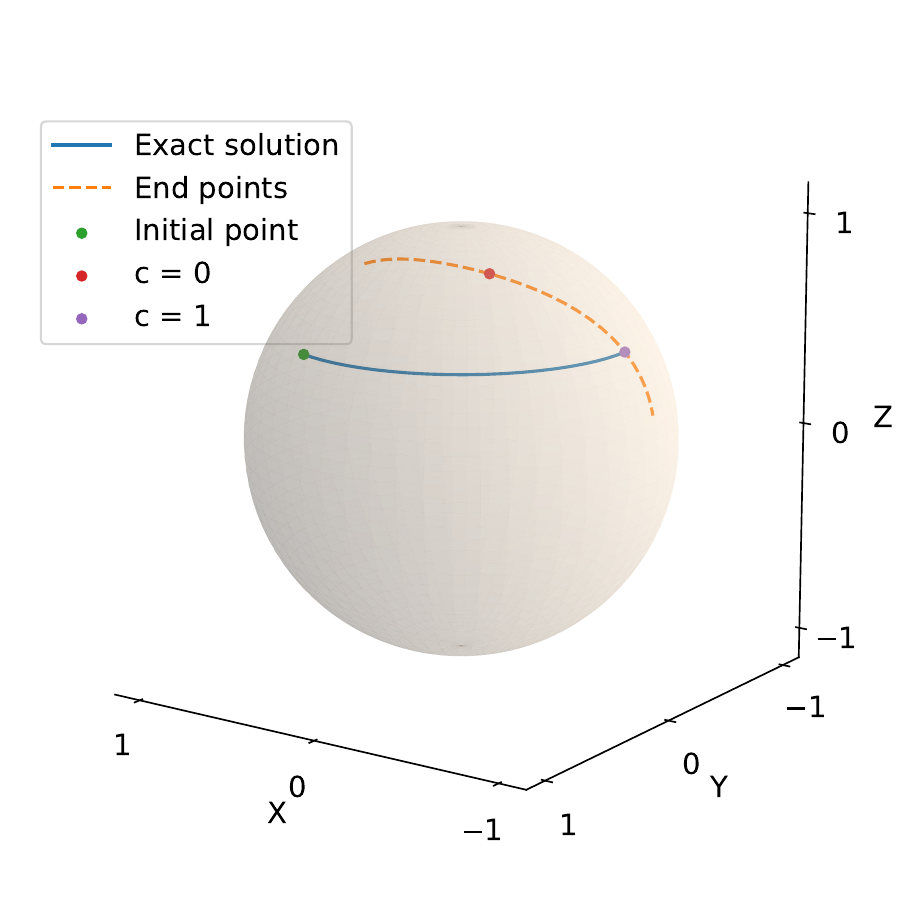}
    \includegraphics[trim = 0 0 0 20, width = 0.57\textwidth]{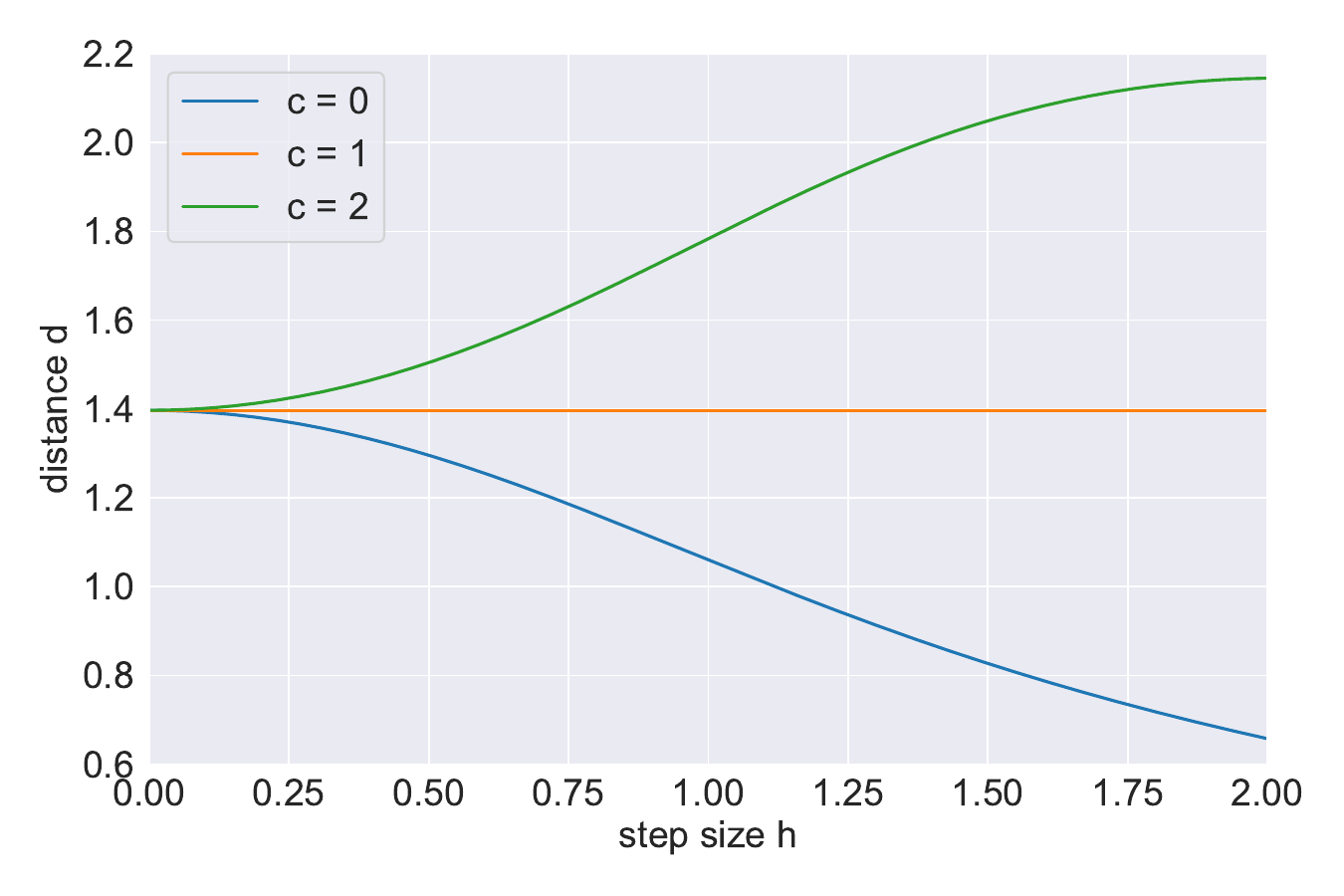}
    \caption{Left: The Implicit Lie-Euler method applied to \eqref{eq:iso_eq} with stepsize $h=2$ and $c\in[-2,2]$. The dashed curve shows the arrival point parametrized by $c$. The solid line depicts the exact solution. Right: the distance between two solutions for increasing stepsizes with three different choices of isotropy parameter $c\in\{0,1,2\}$.}
    \label{fig:isotropy_plot_different_c_values}
\end{figure}
\end{example}

\section{A bound for the global error}\label{sec:glerr}
For the next result, we first consider an initial value problem on the finite-dimensional Riemannian manifold $(M, g),$
\begin{equation}\label{eq:ivp}
\begin{cases}
    \dot{y} = \left.X\right|_y\\ 
    y(0) = y_0 \in M
\end{cases},
\end{equation} 
where $X$ is a smooth vector field, $y_0 \in M$ is the initial value. 
The following theorem is a generalization of Theorem 2 from \cite{celledoni20epm}, where we 
use the 
constant $\nu$ from the monotonicity condition rather than the operator norm of $\nabla X$. 
\begin{theorem}
    Let $(M, g)$ be a Riemannian manifold and fix $y_0 \in M$. Let $\mathcal{U}_{y_0} \subset M$ be a geodesically convex set and $X$ a vector field on $M$ satisfying the monotonicity condition \eqref{eq:X_contractive_in_U} on $\mathcal{U}_{y_0}$ with constant $\nu \in \mathbb{R}$. 
    Let $y(t)=\exp (tX)y_0$ 
    be defined and contained in $\mathcal{U}_{y_0}$ for $t \in [0, t^*], t^* > 0$. 
    Let $\phi_{h,X}$ be a numerical method $y_{j+1}=\phi_{h,X}\left(y_j\right), j = 0, \ldots, k-1$, well defined and contained in $\mathcal{U}_{y_0}$ for any $h$ such that $0 < h \leq h^* \leq t^*, t^* = hk$, whose local error can be bounded for some $p \in \mathbb{N}$ and $C \in \mathbb{R}$ as
    \begin{equation}\label{eq:loc_error}
        d\left(\exp(hX)y, \phi_{h,X}(y)\right) \leq C h^{p+1} \quad \text { for all } y \in \mathcal{U}_{y_0}, h \in (0, h^*].
\end{equation} 
Then, for all $k>0$, the global error is bounded as
\begin{equation}\label{eq:gl_err_est_cases}
    d\left(y(t^*), y_k\right) \leq 
    \begin{cases}
         \frac{C}{\nu}\left(\mathrm{e}^{t^* \nu}-1\right) h^p &\quad \text { for } \nu > 0\\
     Ct^*h^{p} &\quad \text { for } \nu = 0,\\
    \frac{C\mathrm{e}^{-\nu h}}{\nu}\left(\mathrm{e}^{t^*\nu}-1\right) h^p &\quad \text { for } \nu<0
    \end{cases} \quad h \in (0, h^*].
\end{equation} 
\end{theorem}
\begin{proof} Let us denote the global error as $E_k := d\left(y(t^*), y_k\right)$. For $j = 0, \ldots, k-1,$
\begin{align}
    E_{j+1} & \leq d\left(\exp(hX)y(j h), \exp(hX)y_j\right)+d\left(\exp(hX)y_j, \phi_{h,X}\left(y_j\right)\right) \label{eq:first_ineq} \\
    & \leq \mathrm{e}^{h \nu} d\left(y(j h), y_j\right) + d\left(\exp(hX)y_j, \phi_{h,X}\left(y_j\right)\right) \label{eq:second_ineq} \\
    & = \mathrm{e}^{h \nu} E_j + d\left(\exp(hX)y_j, \phi_{h,X}\left(y_j\right)\right) \nonumber\\
    & \leq \mathrm{e}^{h \nu} E_j+C h^{p+1}. \label{eq:fourth_ineq} 
\end{align}
\eqref{eq:first_ineq} is the triangle inequality, where the first term is the error at $j h$ propagated over one step and the second term is the local error. \eqref{eq:second_ineq} is obtained via a Grönwall-type inequality of \cite{kunzinger06gge} for the first term. Using the local error estimate \eqref{eq:loc_error} for the second term we obtain the recursion in \eqref{eq:fourth_ineq}. Considering $t^* = hk$, $\nu \neq 0$  and summing over $j = 1, \ldots, k-1$, we obtain
\begin{equation}\label{eq:e_k_ineq}
E_k \leq C \frac{\mathrm{e}^{t^* \nu}-1}{\mathrm{e}^{h \nu}-1} h^{p+1}.
\end{equation} 
For $\nu > 0$, $\mathrm{e}^{\nu h} - 1 > \nu h$ and \eqref{eq:e_k_ineq} becomes equivalent to the first estimate in \eqref{eq:gl_err_est_cases}. 
For $\nu <0$, $1 - \mathrm{e}^{-\nu h} < \nu h$ and \eqref{eq:e_k_ineq} becomes equivalent to the third estimate in \eqref{eq:gl_err_est_cases}. 
\end{proof}

\begin{remark}
In cases where the monotonicity constant $\nu\ll 0$, in the sense that one can assume $\nu h\rightarrow -\infty$ as $h\rightarrow 0$, see e.g. \cite{hairer96sod}*{Ch IV.15}, one gains an order of convergence such that the global error essentially equals the local error.
\end{remark}

\section{Conclusions and further work}
The notion of B-stability proposed in \cite{butcher1975} for Euclidean spaces has been generalized to Riemannian manifolds.
Building on the work by Simpson-Porco and Bullo \cite{SIMPSONPORCO201474} on contraction systems in Riemannian manifolds, we expressed the B-stability condition in terms of the Riemannian distance function. For this first study, only geodesic versions of the implicit Euler method and the implicit midpoint rule were considered. We proved that in the Riemannian setting, the geodesic implicit Euler method is B-stable for manifolds of non-positive sectional curvature, but not necessarily in positively curved spaces.
Through numerical experiments on the 2-sphere, one finds strong evidence that the GIE method is indeed not B-stable in general.
Another observation was that, contrary to what has been proved in Euclidean spaces, the nonlinear equations associated with the GIE method do not have a unique solution for non-expansive systems. Finally, we showed that the monotonicity constant can be used to obtain improved global error estimates compared to \cites{curry20col, celledoni20epm}.

Many open questions remain for the B-stability properties of numerical methods applied to problems on Riemannian manifolds. There exist many classes of numerical integrators that could be analyzed in this setting. In mechanical engineering, most of the problems of interest are set in manifolds of positive sectional curvature, such as $SO(d), SE(d)\; d=2,3$, $S^2$, $TS^2$ and direct or semidirect products of these. It may also be of interest to consider explicit integrators, in which case B-stability must be replaced by some conditional form of stability, such as the circle contractivity proposed in \cite{dahlquist1979gdo}.

\end{document}